\documentclass[11pt, twoside]{article}
\usepackage{amsfonts}
\usepackage{ulem}

\usepackage{amssymb}
\usepackage{amsmath}
\usepackage{amsthm}
\usepackage{xcolor}
\usepackage{mathrsfs}

\usepackage{verbatim}

\allowdisplaybreaks

\allowdisplaybreaks

\def\rr{{\mathbb R}}

\def\rn{{{\rr}^n}}

\def\Z{{\mathbb Z}}

\def\z+{{\mathbb Z}_+}

\def\supp{{\rm{\,supp\,}}}

\def\r{\right}
\def\lf{\left}

\def\bint{{\ifinner\rlap{\bf\kern.30em--}
\int\else\rlap{\bf\kern.35em--}\int\fi}\ignorespaces}

\def\sbint{{\ifinner\rlap{\bf\kern.32em--}
\hspace{0.078cm}\int\else\rlap{\bf\kern.45em--}\int\fi}\ignorespaces}

\newtheorem{theorem}{Theorem}[section]
\newtheorem{lemma}[theorem]{Lemma}

\theoremstyle{definition}

\newtheorem{remark}[theorem]{Remark}

\newtheorem{definition}[theorem]{Definition}
\numberwithin{equation}{section}

\numberwithin{equation}{section}

\numberwithin{equation}{section}

\begin{document}

\arraycolsep=1pt

\title{\Large\bf New Atomic Decompositions of Weighted Local Hardy Spaces
\footnotetext{\hspace{-0.35cm} {\it 2020
Mathematics Subject Classification}: {42B30, 42B20, 47B37.} 
\endgraf{\it Key words and phrases}: local Hardy space, Calder\'on-Zygmund operators,  atomic decomposition, weight.
\endgraf This project is supported by NSFC (Nos.\,12261083,\,12161083).
\endgraf $^\ast$\,Corresponding author.
}}
\author{Haijing Zhao, Xuechun Yang, Baode Li$^\ast$}
\date{ }
\maketitle

\vspace{-0.8cm}

\begin{center}
\begin{minipage}{13cm}\small
{\noindent{\bf Abstract.}
 We introduce  a new class of weighted local approximate atoms including classical weighted local atoms. Then we further  obtain the  weighted local approximate atomic decompositions  of weighted local Hardy spaces $h_{\omega} ^p(\rn)$  with $0<p\leq 1$ and weight $\omega\in A_1(\rn)$. As  an application, we prove the boundedness of   inhomogeneous Calder\'on-Zygmund  operators on $h_{\omega}^p(\rn)$ via weighted local approximate atoms and molecules.}

\end{minipage}
\end{center}


\section{Introduction\label{s1}}
\quad It is well known that the Hardy space $H ^p (\rn)$ is a good substitute of $L^ p (\rn )$ when $0<p \leq 1$
in pure Fourier analysis \cite{H1972}. However, as Goldberg in \cite{g79} pointed out that $H ^p (\rn )$
space is well suited only to the Fourier analysis, and is not stable under multiplications by the Schwartz test functions. The one reason is that $H ^p (\rn )$ does not contain $\mathcal{S}(\rn)$,
the space of the Schwartz test functions. To circumvent those drawbacks, Goldberg in  \cite{g79}
introduced the local Hardy spaces $h ^p (\rn )$, $0 < p < \infty$. The theory of local Hardy space plays an important role in various fields of
analysis and partial differential equations; see \cite{NWF}, \cite{nwf2} and \cite{ff1993}. In particular, pseudo-differential operators are bounded on local Hardy spaces $h^ p $ for
$0 < p \leq 1$, but they are not bounded on Hardy spaces $H ^p$ for $0 < p \leq 1$; see \cite{g79}.
 Bui \cite{hb81} further studied the weighted version  of the local Hardy space  considered by Goldberg \cite{g79}, where the weight is assumed
to satisfy the condition $(A_\infty)$ of Muckenhoupt.

The theory of  Calder\'on-Zygmund singular integrals began in the 1950s when
 Calder\'on and Zygmund studied convolution operators appeared in elliptic partial differential equations with constant coefficients. To deal with partial differential equations with variable coefficients, Calder\'on and Zygmund introduced the so called second generation of Calder\'on-Zygmund singular integrals.  These operators are not convolution operators, but kernels of
these operators are given by the form $K(x,x-y)$. In 1965, Calder\'on considered the minimal regularity problem of the symbolic calculus, which  opened the door of
the study for the third generation of Calder\'on-Zygmund singular integrals. Precisely, the
operator $T$ is said to be in the {\it third generation of Calder\'on-Zygmund singular integrals} if $T$ is
a continuous linear operator from $\mathcal{S}(\rn)$ to $\mathcal{S}^{\prime}(\rn)$ defined by
$$
\langle Tf, g \rangle=\iint_{{\rn}\times {\rn}}K(x,y)f(y)g(x)dxdy
$$
for all $f$, $g \in \mathcal{S}^(\rn)$  with disjoint supports, where $K(x,y)$, the kernel of $T$, is a locally
integrable function defined on ${{\rn}\times {\rn}}$ away from the diagonal $x = y$. Moreover, $K(x,y)$
satisfies the following conditions: for $\delta \in (0,1]$ and a constant $C > 0$,
\begin{align}\label{e0.66}
|K(x, y)| \leq  \frac{C}{|x-y|^n}  \quad \mathrm {for} \quad x \neq y
\end{align}
and
\begin{align}\label{e01.66}
|K(x, y)-K(x, z)|+|K(y, x)-K(z, x)| \leq C\frac{|y-z|^{\delta}}{|x-z|^{n+\delta}}   \quad \mathrm{for}\,\, \mathrm{all}\,\, |x-z| \geq 2|y-z|.
\end{align}
 It is easy to see that dilations of the kernel, $ \delta ^n K(\delta x,\delta y)$ with $\delta > 0$, satisfy the same
estimates above with the same bounds.
In 2020,  Ding-Han-Zhu \cite{dhz21} introduced {\it inhomogeneous Calder\'on-Zygmund operators $T^{\mu}_{\delta}$}, which the operator  $T^{\mu}_{\delta}$  satisfies similar
conditions as mentioned for the third generation of Calder\'on-Zygmund singular integrals, but
$K(x,y)$ also satisfies a restrictive size condition, namely, the condition \eqref{e0.66} is replaced by
$$
|K(x, y)| \leq C \min \left\{\frac{1}{|x-y|^n}, \frac{1}{|x-y|^{n+\mu}}\right\},  \quad \mathrm{for}\,\, \mathrm{some} \,\, \mu > 0 \,\,\mathrm{and}\,\, x \neq y.
$$
We remark that the above condition for kernels of operators appeared when one studied pseudo-differential operators; see \cite{mc97, ff1993}. They  obtained that  the operator $T^{\mu}_{\delta}$ is  bounded on local Hardy spaces $h^p(\rn)$    with $\max(\frac{n}{n+\mu},\frac{n}{n+\delta})<p\leq1$ when it satisfies  a $n(\frac{1}{p}-1)$-Lip type condition.
The operator  $T^{\mu}_{\delta}$  satisfies similar
conditions  for the third generation of Calder\'on-Zygmund singular integrals, as the same time  the invariant property of dilations  for the third generation of Calder\'on-Zygmund singular integrals does not hold for
inhomogeneous Calder\'on-Zygmund singular integrals. Another difference between these
two kinds of singular integrals is that the third generation of Calder\'on-Zygmund singular
integrals have singularities at the diagonal $x = y$ and the infinity, but inhomogeneous Calder\'on-Zygmund singular integrals have the singularity at the diagonal $x = y$ only. In 2023, Tan-Chen \cite{jcjubuhd23} further obtained that the operator  $T^{\mu}_{\delta}$  is bounded on weighted local Hardy spaces $h^p_{\omega}(\rn)$ with $0<p<\infty$ when it satisfies some cancellation conditions. However, the cancellation conditions  are strict for the operator  $T^{\mu}_{\delta}$ of non-convolution type. In 2022, Dafni-Lau-Picon-Vasconcelos \cite{2wqqk} established  the approximate atomic decompositions of local Hardy spaces, which including classical atomic decompositions of local Hardy spaces. They further obtained the boundedness of the operator  $T^{\mu}_{\delta}$ on local Hardy spaces $h^p(\rn)$ with  $0<p\le1$ when  it satisfies an appropriate cancellation condition   including cancellation condition as special case.

Inspired by \cite{jcjubuhd23} and \cite{2wqqk},  we will introduce weighted local approximate atoms with the   cancellation condition of classical atoms, i.e.,
$$
\int_\rn a(x)x^\alpha dx=0 \quad \ \,\,\mathrm{with}\ |\alpha|\leq s, \,\,s\geq[n(\frac{1}{p}-1)]
$$
replaced by  ball-control conditions , i.e.,
$$
  \begin{cases}\left|\int_{\mathbb{R}^n} a(x) (x-x_B)^{\alpha }d x\right| \leq C\left(\frac{\left|B\right|}{\omega(B)}\right)^{\frac{1}{p}}&\text{if}\quad |\alpha|<\gamma_p:=n(\frac{1}{p}-1)\neq s\\\left|\int_{\mathbb{R}^n} a(x) (x-x_B)^{\alpha }d x\right| \leq C\left(\frac{\left|B\right|}{\omega(B)}\right)^\beta \omega(B)^\eta&\text{if}\quad|\alpha|=\gamma_p=s\end{cases},
$$
where weight $\omega \in A_1(\rn)$ and $\beta:=\eta+\frac{1}{p}$. We will discuss the equivalence of weighted local Hardy spaces defined via approximate atoms  with classical  weighted local Hardy spaces.
The proof idea comes from that of unweighted  case of Dafni-Lau-Picon-Vasconcelos \cite[Proposition 3.3]{2wqqk}. The main difficulty we overcame is that the weighted Lebesgue measure no longer satisfies the translation invariance property of Lebesgue measure, i.e., for any measurable set  $E\subset\rn$ and $x\in \rn$,  $|E|=|E+x|$ but $\omega(E)=\omega(E+x)$ can not be used for weight $\omega$; see the proof of Theorem \ref{m2.36} for more details.
Furthermore, we will introduce approximate  molecules including classical molecules and obtain that they are bounded uniformly on $h^p_w(\rn)$  norm with $0<p\leq1$.
As an application, we will show the operator  $T^{\mu}_{\delta}$ is bounded on $h^p_w(\rn)$ when it satisfies  ball-control conditions, which is weaker than the classical cancellation condition.

This article is organized as follows. In Sect. 2, we  introduce weighted local approximate atoms, i.e.,  the   cancellation condition of classical atoms replaced by  ball-control conditions (see Definition \ref{d2.33}). And we discuss the equivalence of weighted local Hardy spaces defined via approximate atoms  with classical  weighted local Hardy spaces (see Theorem \ref{m2.36}). In Sect. 3, we  introduce approximate  molecules, i.e.,  the   cancellation condition of classical molecules replaced by  ball-control conditions including classical molecules (see Definition \ref{d2.11}).  Then we  obtain that they are bounded uniformly on $h^p_w(\rn)$  norm with $0<p\leq1$ (see Theorem \ref{d3.66}).  As an application, in Sect. 4, we obtain  the boundedness of the operator  $T^{\mu}_{\delta}$  on $h^p_w(\rn)$ with $\omega\in A_1(\rn)$ and $0<p\leq 1$.

It  leaves an open problem:  is it still holds true for the approximate atomic decompositions of $h^p_\omega(\rn)$ when weight $\omega\in A_1(\rn)$ replaced by $\omega\in A_\infty(\rn)?$

\section*{ Notations}
\ \,\ \ $-Z_+:$$=\{1,2,3 \cdots \}$;

$-\mathbb{N}:$$ =\{ 0, 1, 2, 3, \cdots\}$;

$-C:$ a positive constant which is independent of the main parameters but may vary from line to line;

$-\lesssim:$ $f\le Cg$ is denoted by $f\lesssim g$;

$-q':$ The conjugate exponent of $q$ with $q\in[1,\infty]$, namely, $1 / q+1 / q ^{\prime}=1$;

$-[t]:$ the biggest integer no more than $t\in \mathbb{R}$;

$-\mathbb{R}^{n}:$ $n$-dimensional Euclidean space;

$-B(x_B,r):$$=\{x:|x-x_{B}|<r\}$;

$-B:$ a  ball in \ $\mathbb{R}^n$;

$-|B|:$ the Lebesgue measure of $B$;

$-\chi_{E}:$ the characteristic function of $E$;

$-\mathcal{S}(\rn):$ the space of Schwartz functions;

$-\mathcal{S}^{\prime}(\rn):$ the space of tempered distributions.




\section{$h_w^p (\mathbb{R}^n )=h_{w,approx}^{p, q, s} (\mathbb{R}^n)$}
Throughout this article, a weight $\omega $ is a non-negative locally integrable function on $\rn $, we shall denote $\omega(E):=\int_E \omega(x) d x$ for any set $E \subset \mathbb{R}^n$.
We first recall the weight class of Muckenhoupt.
\begin{definition}\label{d2.4}
 For $1<q<\infty$, a locally integrable function $\omega: \mathbb{R}^n \rightarrow[0, \infty)$ is said to be an {\it $ A_q$ weight } if
$$
[\omega]_{A_q}:=\sup _B\left(\frac{1}{|B|} \int_B \omega(x) \mathrm{d} x\right)\left(\frac{1}{|B|} \int_B \omega(x)^{\frac{1}{1-q}} \mathrm{~d} x\right)^{q-1}<\infty .
$$
A locally integrable function $\omega: \mathbb{R}^n \rightarrow[0, \infty)$ is said to be an {\it $ A_1$ weight }  if
$$
[\omega]_{A_1}:=\sup _B \frac{1}{|B|}\left(\int_B \omega(x) \mathrm{d} x\right){\underset{x \in B}{\operatorname{ess} \sup }[\omega(x)]^{-1}}<\infty.
$$
Define $ A_\infty(\rn):=\bigcup_{q\in[1,\infty)}A_q(\rn).$

For any given $\omega \in A _q(\rn)$, define the {\it critical index of $\omega$ } by
\begin{align*}
q_\omega: =\inf\{q\in[1,\infty ): \omega\in A_q(\rn)\}.
\end{align*}
\end{definition}
Given a weight function $\omega\in A_\infty(\rn)$. For $1\le q<\infty$, the {\it weighted Lebesgue space} $L^q_\omega(\rn)$ is the space of functions $f$ such that
$$\|f\|_{L^q_\omega}:=\lf(\int_{\rn} |f(x)|^q\omega(x)\,dx\r)^{\frac{1}{q}}<\infty.$$

\begin{lemma}\label{l3.16}
\begin{enumerate}
\item [(i)]\cite[Theorem 2.1(iv)]{at86} Let $q\geq1$, $w\in A_{q}(\rn)$ and $k\in Z_{+}$. Then
 \begin{align}\label{e2.16}
2^{\frac{kn}{q}}\omega(B(x_{B}, r))\leq \omega(B(x_{B},2^{k}r))\leq C2^{nqk}\omega(B(x_{B}, r)),
\end{align}
where $C$ is a positive constant independent of $r$, $q$ and $x_{B}$.
\item [(ii)]
\cite[Lemma 1.1]{jcc85}  Let $q\in[1, \infty)$. If $\omega\in A_{q}(\rn)$,
then
$$\frac{1}{|B(x_{B},r)|}\int_{B(x_{B},r)}|f(x)|dx\leq\left(\frac{C}{\omega(B(x_{B},r))}\int_{B(x_{B},r)}|f(x)|^{q}\omega(x)dx\right)^\frac{1}{q},
$$
\end{enumerate}
where $C$ is a positive constant  independent of $r$ and $x_{B}$.
\end{lemma}

The weighted local Hardy spaces, introduced by Bui \cite {hb81}, can be given by several definitions. As follows, we choose to use the definition via maximal function.
\begin{definition}\label{d2.5}
Let $ \omega \in A_{\infty}(\rn) $ and let $\phi$ be a fixed Schwartz function in $\mathcal{S}(\rn)$ such that $\int_{\rn}\phi(x)dx\neq 0$, and set $\phi_{t}(x):=t^{-n}\phi(t^{-1}x)$. Then we say that a tempered distribution $f \in \mathcal{S}^{\prime}(\rn)$ lies in  $ h^p_\omega(\rn)$ if $\|f\|_{h^p_\omega}=\left\|m_\phi f\right\|_{L^p_\omega}<\infty$,
where
\end{definition}
$$
\begin{aligned}
m_{\phi}f(x):= \sup\limits_{0<t<1}|f\ast\phi_{t}(x)|\in{L^p_\omega(\rn)}, \quad x\in\rn.
\end{aligned}
$$

The spaces are independent of the choice of $\phi$, and different choices give comparable $\|\cdot\|_{h_{\omega}^{p}}$. For $p\geqslant1$ this defines a norm making $h_{\omega}^{p}(\rn)$ a Banach space and when $p>1$ it is equal to $L_{\omega}^{p}(\rn)$ with equivalent norms. For $0<p\leq1$, the space $h_{\omega}^{p}(\rn)$ is a complete metric space with the distances $d(f,g)=\parallel f-g\parallel_{h_{\omega}^{p}}^p$. Although $h_{\omega}^{p}(\rn)$ is not locally convex for $0<p<1$ and $\| f\|_{h_{\omega}^{p}}$ is a quasi-norm, we will still refer to the functional $\| f\|_{h_{\omega}^{p}}$ as the $h_{\omega}^{p}$ norm for simplicity.

%

Give $0<p\leq 1$,  let $$\gamma_p:=n\lf(\frac{1}{p}-1\r), s:=[\gamma_p].$$
Approximate atoms are introduced by Dafni-Lau-Picon-Vasconcelos  \cite{2wqqk}, weighted approximate atoms are further introduced in the following Definition \ref{d2.33}(ii). Instead, our ball-control conditions intrinsically related to the value of
$0 < p \leq 1$ in the following way: if $p \neq\frac{n}{n+k}$ for every $k\in\Z_{+}$, namely $s\neq \gamma_p$, it suffices to bound the moments up order $|\alpha|\leq s$ by a ball-control condition; on the other hand, if $p=\frac{n}{n+k}$, another a ball-control condition is needed when $|\alpha|=s=\gamma_p$.
\begin{definition}\label{d2.33}
\begin{enumerate}
\item [(i)]
Let $0<p\leq 1\leq q\leq\infty$ and $p\neq q $ such that $\omega \in A_q(\rn)$ with critical index $q_\omega$. For $s_0\in \mathbb{N}$ with $s_0\geq[n(\frac{q_\omega}{p}-1)]$,
a real-valued function  $a$  on $\rn$ is  called a {\it$(p,q, s_0, \omega )$-atom} if it satisfies the following ($A_1$), ($A_2$) and ($A_3$):
\begin{enumerate}
  \item [$(A_1)$]$\supp(a)\subset B(x_B, r)$,
\item [$(A_2)$]$\|a\|_{L_{\omega}^{q}}\leq \omega(B(x_{B}, r))^{\frac{1}{q}-\frac{1}{p}}$,
\item [$(A_3)$] if $r<1$, $\int_\rn a(x)x^\alpha dx=0$ for every multi-index $\alpha$ with $|\alpha|\leq s$, or if $r\geq1$, then $a$ does not have any vanishing moment.
\item [(ii)]A real-valued function $a$ on $\rn$ is call a {\it $(p, q, s,\omega)$-approximate atom } with $\omega\in A_1(\rn)$ and $s:=[\gamma_p]$ if there exists a ball $B(x_B,r) \subset \mathbb{R}^n$  and $\eta>0$ such that $a$ satisfies ($A_1$), ($A_2$)  and  the following ($A\acute{_3}$) when  $ r<1 $:
   \item  [$(A\acute{_3})$]$   \begin{cases}\left|\int_{\mathbb{R}^n} a(x) (x-x_B)^{\alpha }d x\right| \leq C\left(\frac{\left|B\right|}{\omega(B)}\right)^{\frac{1}{p}}&\text{if}\quad |\alpha|<\gamma_p\neq s\\\left|\int_{\mathbb{R}^n} a(x) (x-x_B)^{\alpha }d x\right| \leq C\left(\frac{\left|B\right|}{\omega(B)}\right)^\beta \omega(B)^\eta&\text{if}\quad|\alpha|=\gamma_p=s\end{cases}$,\\
   where $\beta:=\eta+\frac{1}{p}$, or
   $a$ satisfies $(A_1)$ and $(A_2)$ when  $ r\geq1 $.
\end{enumerate}
\end{enumerate}
\end{definition}

\begin{remark}\label{re1.1}
\begin{enumerate}
\item [(i)]When  $C=0$, the $(p, q,s, \omega)$-approximate atom is also a $(p,q, s,\omega)$ atom as in  Definition $\ref{d2.33}$(i).
\item [(ii)] Obviously, $(A_3)$ implies $(A\acute{_3})$ and hence any local weighted $(p, q, s, \omega)$-atom $a$ is also a local weighted $(p,q,s, \omega)$-approximate atom with $\omega \in A_1(\rn)$.
\item [(iii)]When $\omega$ $\equiv 1$, the $(p, q,s, \omega)$-approximate atom is also a $(p,q,\omega)$ approximate atom as in \cite[Definition 3.2]{2wqqk}.
\item [(iv)] If $r>1$, ball-control conditions  of $(A\acute{_3})$ can be deduced by ($A_1$) and ($A_2$). In fact,
by H\"older's  inequality, $(A_1)$ and $(A_2)$ of Definition \ref{d2.33} and Definition \ref{d2.4},  we obtain
\begin{align}\label{e1.1}
&\left|\int_{\mathbb{R}^n} a(x) (x-x_B)^{\alpha }d x\right| \\\nonumber
&\qquad\leq r^{|\alpha|}\int_{B}|a(x)|\omega(x)^{\frac{1}{q}}\omega(x)^{-\frac{1}{q}}dx\\\nonumber
&\qquad\leq r^{|\alpha|}\left(\int_{B} |a(x)|^q \omega(x)dx\right)^{\frac{1}{q}} \left(\int_{B} \omega(x)^{-\frac{1}{q-1}}dx\right)^{1-\frac{1}{q}}\\\nonumber
&\qquad\leq [\omega]_{A_q}r^{|\alpha|}\omega(B)^{\frac{1}{q}-\frac{1}{p}}\frac{|B|}{\omega(B)^{\frac{1}{q}}}\\\nonumber
&\qquad\leq [\omega]_{A_q}r^{|\alpha|+n}{\omega\left(B\left(x_B, r\right)\right)}^{-\frac{1}{p}}.
\end{align}

Furthermore, when $|\alpha|<\gamma_p$, by \eqref {e1.1} and $r>1$, we obtain
\begin{align}\label{e2.112}
&\left|\int_{\mathbb{R}^n} a(x) (x-x_B)^{\alpha }d x\right|
\leq [\omega]_{A_q}r^{|\alpha|+n}{w\left(B\left(x_B, r\right)\right)}^{-\frac{1}{p}}\\\nonumber
&\quad\quad\leq [\omega]_{A_q}r^{n+n(\frac{1}{p}-1)}{\omega\left(B\left(x_B, r\right)\right)}^{-\frac{1}{p}}
\lesssim\left(\frac{\left|B\right|}{\omega(B)}\right)^{\frac{1}{p}}\nonumber.
\end{align}

When $|\alpha|=\gamma_p$ and $\beta=\eta+\frac{1}{p}$, by \eqref {e1.1}, $\gamma_p=n(\frac{1}{p}-1)$ and $r>1$,  we have
\begin{align}\label{e2.113}
&\left|\int_{\mathbb{R}^n} a(x) (x-x_B)^{\alpha }d x\right|
\leq[\omega]_{A_q} r^{|\alpha|+n}{\omega \left(B\left(x_B, r\right)\right)}^{-\frac{1}{p}}\\\nonumber
&\quad\quad\leq [\omega]_{A_q}r^{n+n(\frac{1}{p}-1)}{\omega\left(B\left(x_B, r\right)\right)}^{-\frac{1}{p}}
\leq [\omega]_{A_q}r^\frac{n}{p}{\omega (B\left(x_B, r\right))}^{-\frac{1}{p}}\\\nonumber
&\quad\quad\lesssim\left(\frac{\left|B\right|}{\omega(B)}\right)^\beta \omega(B)^\eta\nonumber.
\end{align}
Combining the above estimates, we obtain that $(A\acute{_3})$  holds when $r>1$.
\end{enumerate}
\end{remark}
\begin{definition}
\begin{enumerate}
\item [(i)]Let $\omega \in A_{\infty}$ and $0<p\leq 1\leq q\leq{\infty}$, $s_0\geq[n(\frac{q_\omega}{p}-1)]$. The {\it weighted atomic local Hardy space $h_\omega^{p, q, s_0}\left(\mathbb{R}^n\right)$} is defined to be the set of all $f \in \mathcal{S}^{\prime}\left(\mathbb{R}^n\right)$ satisfying that $f=\sum_{i=0}^{\infty} \lambda_i a_i$ in $\mathcal{S}^{\prime}\left(\mathbb{R}^n\right)$, where $\left\{\lambda_i\right\}_{i \in \mathbb{N}_0} \subset \mathbb{R}, \sum_{i=0}^{\infty}\left|\lambda_i\right|^p<\infty$ and $\left\{a_i\right\}_{i \in \mathbb{N}}$ are $(p, q, s_0, \omega)$-atoms. Moreover, the quasi-norm of $f \in h_\omega^{p, q, s_0}\left(\mathbb{R}^n\right)$ is defined by
$$
\|f\|_{h_\omega^{p, q, s_0}} \equiv \inf \left\{\left[\sum_{i=0}^{\infty}\left|\lambda_i\right|^p\right]^{1 / p}\right\},
$$
where the infimum is taken over all the decompositions of $f$ as above.
\item [(ii)]When local weighted $(p,q,s_0,\omega)$-atoms replaced by  local weighted $(p,q,s,\omega)$-approximate atoms with $\omega\in A_1$,
$h^{p,q,s}_{w,approx}(\rn)$ can be also defined as $h^{p,q,s_{0}}_w(\rn)$.
\end{enumerate}
\end{definition}

\begin{theorem}\label{m2.36}
Let $\omega \in A_1(\rn)$, $0<p\leq1\leq q<\infty$ with $p\neq q$ and $s:=[n(\frac{1}{p}-1)]$. Then there exist $C>0$ such that, for any $(p,q,s,\omega)$-approximate atom $a$, $\|a\|_{h_w^{p}}\leq C.$
\end{theorem}
\begin{proof}
Let  $\phi \in \mathcal{S}(\rn)$ with $\int_{\rn} \phi(x)dx \neq 0$. Suppose $\supp(a) \subset B(x_{B}, r)$.
Let us write
$$
\begin{aligned}
\|a\|_{h_w^{p}}=\|m_\phi a\|_{L_{\omega}^{p}}^p
\leq \int_{ 2B}\left(m_\phi a(x)\right)^p\omega(x)dx
+\int_{ (2B)^c}\left(m_\phi a(x)\right)^p\omega(x)dx=:\mathrm{S}_1+\mathrm{S}_2.
\end{aligned}
$$

Let us first estimate $\mathrm{S}_1$. When $q>1$, since  $\frac{q}{p} > 1$, by H\"older's inequality, the boundedness  of $m_\phi f$  on $L^q_\omega\left(\mathbb{R}^n\right)$ ($1< q< \infty$) and  Lemma \ref{l3.16}(i), we have
$$
\begin{aligned}
&\int_{ 2B}\left(m_\phi a(x)\right)^p\omega(x)dx\\
& \leq\left(\int_{ 2B}\left[{(m_\phi a)}^p\omega(x)^\frac{p}{q}\right]^\frac{q}{p}dx\right)^\frac{p}{q}\left(\int_{ 2B}\left[\omega(x)^\frac{q-p}{q}\right]^\frac{q}{q-p}dx\right)^\frac{q-p}{q}\\
& \le \left(\int_{ 2B}{(m_\phi a)}^q\omega(x)dx\right)^\frac{p}{q}\left(\int_{2B}\omega(x)dx\right)^\frac{q-p}{q}\\
& \le C_\phi \|a\|_{L_{\omega}^{q}(B)}^p \omega(2B)^{\frac{q-p}{q}}\\
& \leq C_{\phi, q, p, n}.
\end{aligned}
$$

Furthermore, when $q=1$, from the weak type $(L^1_\omega,\, L^1_\omega)$ of Hardy Littlewood-Paley maximal operator $M$ with $\omega\in A_1(\rn)$, the definition of $A _1 (\rn)$ and  $R:=\omega(B)^{-\frac{1}{p}}$, we deduce that
$$
\begin{aligned}
& \int_{2B}\left(m_\phi a(x)\right)^p\omega(x)dx
\leq C_\phi\int_{2 B}\left[ M a(x)\right]^p \omega(x)d x\\
&\quad\leq C_\phi p\int_{0}^\infty \lambda^{p-1}\omega(\{x\in2B: Ma(x)>\lambda\})d\lambda\\
&\quad \leq C_\phi p\int_{0}^R \lambda^{p-1}\omega(2B)+p\int_{R}^\infty \lambda^{p-1}\omega(\{x\in2B: Ma(x)>\lambda\})d\lambda\\
&\quad \leq C_\phi p\int_{0}^R \lambda^{p-1}\omega(2B)+p\int_{R}^\infty \lambda^{p-1}\frac{C}{\lambda}\int_{\rn}|a(x)|\omega(x)dx d\lambda\\
&\quad \leq C_\phi p\int_{0}^R \lambda^{p-1}\omega(2B)+Cp\int_{R}^\infty \lambda^{p-2}d\lambda \omega(B)^{1-\frac{1}{p}}\\
&\quad \leq C_\phi R^p 2^n \omega(B)+ \frac{Cp}{1-p}R^{p-1}\omega(B)^{1-\frac{1}{p}}
 \leq C_\phi 2^n+\frac{Cp}{1-p}=C_{\phi, s, p, n }.
\end{aligned}
$$

Next, let us now estimate $\mathrm{S}_2 $.
%
Let $s:=[\gamma_p]$. Using the Taylor expansion of $\phi$ up to order $s-1$ :
\begin{align}\label{e4}
\phi_t(x-y)=\sum_{|\alpha| \leq s-1} C_\alpha t^{-n-|\alpha|}\phi^{(\alpha)}\left(\frac{x-x_B}{t}\right)(x_B-y )^\alpha+ \mathcal{R}_{s}(\xi ),
\end{align}
where
\begin{align*}
\phi^{(\alpha)}(x)=\frac{\partial^{\alpha} \phi\left(x_1, \ldots, x_n\right)}{\partial x_1^{\alpha_1} \ldots \partial x_n^{\alpha_n}},\quad
\xi:=(x-x_B)+\theta(x_B-y),\quad \theta\in (0,1),
\end{align*}
and
\begin{align*}
\mathcal{R}_{s}(\xi)=\sum_{|\alpha|=s} C_\alpha t^{-n-|\alpha|} \phi^{(\alpha)}\left(\xi\right) (x_B-y)^\alpha.
\end{align*}

Since $y \in B$ and $x \notin 2 B$ and $\xi$ lies on the line between $x-x_B$ and $x_B-y$,  we have $\left|\xi\right| \geq\frac{|x-x_B|}{ 2}$.
Thus, by \eqref{e4}, we obtain
\begin{align}\label{e6}
&\left|\phi_t* a(x)\right|
=\left|\int_{\rn}\left(\phi_t(x-y)- \sum_{|\alpha| \leq s-1} t^{-n-|\alpha|}\phi^{(\alpha)}\left(\frac{x-x_B}{t} \right)(x_B-y)^\alpha\right) a(y)dy\right|\nonumber\\
&\quad +\left|\int_{\rn}\sum_{|\alpha| \leq s-1} t^{-n-|\alpha|}\phi^{(\alpha)}\left(\frac{x-x_B}{t} \right)  a(y)(x_B-y)^\alpha d y\right| \\\nonumber
&\lesssim  \sum_{|\alpha| \leq s-1} t^{-n-|\alpha|}\left|\phi^{(\alpha)}\left(\frac{x-x_B}{t} \right)\right| \left|\int_{B} a(y)(x_B-y)^\alpha d y\right| \\\nonumber
& + \sum_{|\alpha| =s} t^{-n-|\alpha|}\left|\int_{B} \left|\mathcal{R}_{s}(\xi)\right| |a(y)| d y\right|.\nonumber
\end{align}

Let us estimate $\mathrm{S}_2$ under two cases: $0<r<1$ and $r\geq1$.

$\textbf{Case I}$: \textbf{estimate} $\mathrm{S}_2$ \textbf{under} $0<r<1$. Let $k_0\in Z_+$ satisfying $2^{k_0}r\leq2< 2^{k_0+1}r$. For $\phi \in \mathcal{S}(\rn)$, we will use the bound $\left| \phi^{(\alpha)}(x)\right| \leq C_{\alpha}|x|^{-N_\alpha}$, where $N_\alpha > 0$ depending on $\alpha$, will be
chosen conveniently.
We take $N_\alpha= n + |\alpha|$ for the first part and  $N_\alpha = n + |\alpha|+l$ for the second one, where $l\in Z_+$ to be fixed later. Let $A_k :=\{2^kr \leq\left|x-x_B\right| < 2^{k+1}r\}$, $k=1,2,3, \cdots$. By (\ref{e6}), we have
\begin{align}\label{e7}
&\int_{ (2 B)^c}\left(\sup _{0<t<1}\left|\phi_t * a(x)\right|\right)^p \omega(x)d x \\\nonumber
&\lesssim \int_{(2 B)^c} \left(\sup _{0<t<1}\sum_{|\alpha| \leq s-1} t^{-n-|\alpha|}\left|\phi^{(\alpha)}\left(\frac{x-x_B }{t}\right)\right| \left|\int_{B} a(y)(x_B-y)^\alpha d y\right|\right)^p\omega(x)dx \\\nonumber
&\quad +\int_{(2 B)^c}\left( \sup _{0<t<1}\sum_{|\alpha| =s}  t^{-n-|\alpha|}\left|\int_{B} \left|\mathcal{R}_{s}(\xi)\right| |a(y)| d y\right|\right)^p \omega(x)d x\\\nonumber
&\lesssim \sum_{k=1}^{k_0-1}\int_{A_k} \sup _{0<t<1}\sum_{|\alpha| \leq s} t^{-np-\left|\alpha\right| p}\left|\frac{x-x_B}{t}\right|^{-np-\left|\alpha\right|p }\left|\int _{\mathbb{R}^n}a(y)(x_B-y)^\alpha d y\right|^p  \omega(x)d x\\\nonumber
&\quad +\sum_{k=k_0}^{\infty}\int_{A_k}\sup _{0<t<1}\sum_{|\alpha| \leq s}  t^{-np-\left|\alpha\right| p}\left|\frac{x-x_B}{t}\right|^{-np-\left|\alpha\right|p -l p}\left|\int _{\mathbb{R}^n}a(y)(x_B-y)^\alpha d y\right|^p  \omega(x)d x\\\nonumber
& =:\mathrm{I_1+I_2}.\nonumber
\end{align}

First, let us estimate $\mathrm{I_1}$. When $|\alpha|\leq s \neq\gamma_p$ and  $r<1$, by Definiton \ref{d2.33}$(A\acute{_3})$, Lemma \ref{l3.16}(i), $\omega\in A_1(\rn )$ and $n-np-|\alpha|p>0$, we obtain
\begin{align*}
\mathrm{I_1}
&\lesssim \sum_{k=1}^{k_0-1} \sum_{|\alpha| \leq s} \int_{A_k  }\left|x-x_B\right|^{-np -\left|\alpha\right| p}\left| \int_{\mathbb{R}^n} a(y)(x_B-y)^\alpha d y\right|^p \omega(x)d x \\\nonumber
& \lesssim \sum_{k=1}^{k_0-1} \sum_{|\alpha| \leq s} \int_{A_k  }\left(2^{k}r\right)^{-np -\left|\alpha\right| p} r^n \omega(B(x_B,r))^{-1 }  \omega(x)d x \\\nonumber
& \lesssim \sum_{k=1}^{k_0-1}\sum_{|\alpha| \leq s}2^{-k(np +\left|\alpha\right| p)} 2^{kn} r^{n-np-|\alpha| p}\\\nonumber
& = C \sum_{|\alpha| \leq s}\frac{2^{(n-np-|\alpha| p)}(1-2^{(n-np-|\alpha| p)(k_0-1)})}{1-2^{(n-np-|\alpha| p)}}r^{n-np-|\alpha| p}\\\nonumber
& =C\sum_{|\alpha| \leq s} \frac{2^{(n-np-|\alpha| p)k_0}-2^{(n-np-|\alpha| p)}}{2^{(n-np-|\alpha| p)}-1}r^{n-np-|\alpha| p}\\\nonumber
& \lesssim \sum_{|\alpha| \leq s}\frac{r^{-(n-np-|\alpha| p)}-2^{(n-np-|\alpha| p)}}{2^{(n-np-|\alpha| p)}-1}r^{n-np-|\alpha| p}\\\nonumber
& =C\sum_{|\alpha| \leq s}\frac{1-2r^{(n-np-|\alpha| p)}}{2^{(n-np-|\alpha| p)}-1}\\\nonumber
&= C .
\end{align*}

 Next, let us estimate $\mathrm{I_2}$. When $|\alpha|\leq s\neq\gamma_p$ and $r<1$,  by Definiton \ref{d2.33}$(A\acute{_3})$, Lemma \ref{l3.16}(i), $\omega\in A_1(\rn)$ and $2^{k_0}\sim \frac{1}{r}$, we obtain
\begin{align*}
\mathrm{I_2}
& \lesssim\sum_{|\alpha| \leq s}\sum_{k=k_0}^{\infty} \int_{A_k }\sup _{0<t<1}t^{-np-\left|\alpha\right| p}\left|\frac{x-x_B}{t}\right|^{-np-\left|\alpha\right|p -l p}\left|\int _{\mathbb{R}^n}a(y)(x_B-y)^\alpha d y\right|^p \omega(x)d x\\\nonumber
&\lesssim \sum_{|\alpha| \leq s} \sum_{k=k_0}^{\infty} \int_{A_k }\left|x-x_B\right|^{-np -\left|\alpha\right| p-{l p}}\left| \int_{\mathbb{R}^n} a(y) (x_B-y)^\alpha d y\right|^p \omega(x)d x \\\nonumber
&\lesssim \sum_{|\alpha| \leq s} \sum_{k=k_0}^{\infty} \int_{A_k }\left(2^kr\right)^{-np -\left|\alpha\right| p-{l p}}r^n w(B(x_B,r))^{-1 } \omega(x)d x \\\nonumber
&\lesssim \sum_{|\alpha| \leq s} \sum_{k=k_0}^{\infty} \left(2^kr\right)^{-np -\left|\alpha\right| p-{l p}}r^n 2^{kn} \\\nonumber
&\lesssim \sum_{|\alpha| \leq s} r^{n-np-|\alpha|p-lp}\frac{2^{k_0(n-np-|\alpha|p-lp)}}{1-2^{n-np-|\alpha|p-lp}}\\\nonumber
&\lesssim \sum_{|\alpha| \leq s}r^{n-np-|\alpha|p-lp}\frac{r^{-(n-np-|\alpha|p-lp)}}{1-2^{n-np-|\alpha|p-lp}} \\\nonumber
&\lesssim \sum_{|\alpha| \leq s}\frac{1}{1-2^{n-np-|\alpha|p-lp}}\lesssim 1\nonumber,
\end{align*}
 where we pick $l\in Z_+$ such that $n-np-s-lp<0$.

When $|\alpha|= s =\gamma_p$ and $r<1$,  by Definition \ref{d2.33}$(A\acute{_3})$, Lemma \ref{l3.16}(i), $np+\left|\alpha\right|p=n$, $\omega\in A_1(\rn)$ and $2^{k_0}\sim \frac{1}{r}$. we have
\begin{align*}
\mathrm{I_1}
& \lesssim\sum_{k=1}^{k_0-1}\sum_{|\alpha| = s}\int_{A_k }\sup _{0<t<1}t^{-np-\left|\alpha\right| p}\left|\frac{x-x_B}{t}\right|^{-np-\left|\alpha\right|p }\left|\int _{\mathbb{R}^n}a(y)(x_B-y)^\alpha d y\right|^p \omega(x)d x\\\nonumber
&\lesssim \sum_{k=1}^{k_0-1} \sum_{|\alpha|= s} \int_{A_k }\left|x-x_B\right|^{-np -\left|\alpha\right| p}\left| \int_{\mathbb{R}^n} a(y)(x_B-y)^\alpha d y\right|^p \omega(x)d x \\\nonumber
&\lesssim \sum_{k=1}^{k_0-1} \sum_{|\alpha| = s} \int_{A_k}\left|x-x_B\right|^{-np -\left|\alpha\right| p} \left(\frac{\left|B\right|}{\omega(B)}\right)^{\beta p } \omega(B)^{\eta p} \omega(x)d x \\\nonumber
&\lesssim\sum_{k=1}^{k_0-1} \sum_{|\alpha| = s} \int_{A_k }\left(2^kr\right)^{-np -\left|\alpha\right| p} |B|^{\eta p+1} \omega(B)^{-1}\omega(x)d x \\\nonumber
&\lesssim \sum_{k=1}^{k_0-1}  (2^k)^{-n} 2^{kn} r^{\eta pn}\lesssim \sum_{k=1}^{k_0-1}r^{\eta pn}\lesssim (k_0-1)r^{\eta pn}\lesssim 1.\nonumber
\end{align*}

Next, let us estimate $\mathrm{I_2}$. When $|\alpha|=s=\gamma_p$ and $r<1$, by Definition \ref{d2.33}$(A\acute{_3})$, Lemma \ref{l3.16}(i), $np+\left|\alpha\right|p=n$, $\omega\in A_1(\rn)$ and $2^{k_0}\sim \frac{1}{r}$,  we obtain
\begin{align*}
 \mathrm{I_2}
&\lesssim \sum_{k=k_0}^{\infty}  \sum_{|\alpha| = s}\int_{A_k }\sup _{0<t<1}t^{-np-\left|\alpha\right| p-lp}\left|\frac{x-x_B}{t}\right|^{-np-\left|\alpha\right|p -lp}\left|\int _{\mathbb{R}^n}a(y)(x_B-y)^\alpha d y\right|^p \omega(x)d x\\\nonumber
&\lesssim \sum_{k=k_0}^{\infty} \sum_{|\alpha|= s} \int_{A_k}\left|x-x_B\right|^{-np -\left|\alpha\right| p-lp}\left| \int_{\mathbb{R}^n} a(y)(x_B-y)^\alpha d y\right|^p \omega(x)d x \\\nonumber
&\lesssim \sum_{k=k_0}^{\infty} \sum_{|\alpha| = s} \int_{A_k}\left|x-x_B\right|^{-np -\left|\alpha\right| p-lp} \left(\frac{\left|B\right|}{\omega(B)}\right)^{\beta p } \omega(B)^{\eta p} \omega(x)d x \\\nonumber
&\lesssim \sum_{k=k_0}^{\infty}  \int_{A_k }\left(2^kr\right)^{-n-lp} r^{\eta pn+n}\omega(B)^{-1} \omega(x)d x \\\nonumber
&\lesssim \sum_{k=k_0}^{\infty}  (2^k)^{-n-lp} 2^{kn} r^{n\eta p-lp}\\\nonumber
&\lesssim\sum_{k=k_0}^{\infty}2^{-klp}r^{\eta pn-lp}
\lesssim \frac{2^{k_0 lp}}{1-2^{lp}}r^{\eta pn-lp}\lesssim \frac{r^{\eta pn}}{1-2^{lp}}\lesssim1.\nonumber
\end{align*}

\textbf{Case II}: $\textbf{estimate }$ $\textbf{$\mathrm{S}_2$}$ $\textbf{under}$ $\textbf{$r\geq1$}$.  When $|\alpha|\leq s \neq\gamma_p$, where $l\in Z_+$ to be fixed later.  By Definiton \ref{d2.33}$(A\acute{_3})$ and Lemma \ref{l3.16}(i), we have
\begin{align*}
\mathrm{I_2}
& \lesssim\sum_{|\alpha| \leq s}\sum_{k=1}^{\infty} \int_{A_k }\sup _{0<t<1}t^{-np-\left|\alpha\right| p}\left|\frac{x-x_B}{t}\right|^{-np-\left|\alpha\right|p -l p}\left|\int _{\mathbb{R}^n}a(y)(x_B-y)^\alpha d y\right|^p \omega(x)d x\\\nonumber
&\lesssim \sum_{|\alpha| \leq s} \sum_{k=1}^{\infty} \int_{A_k }\left|x-x_B\right|^{-np -\left|\alpha\right| p-{l p}}\left| \int_{\mathbb{R}^n} a(y) (x_B-y)^\alpha d y\right|^p \omega(x)d x \\\nonumber
&\lesssim\sum_{|\alpha| \leq s} \sum_{k=1}^{\infty} \int_{A_k}\left(2^kr\right)^{-np -\left|\alpha\right| p-{l p}}r^n w(B(x_B,r))^{-1 } \omega(x)d x \\\nonumber
&\lesssim\sum_{|\alpha| \leq s} \sum_{k=1}^{\infty} \left(2^kr\right)^{-np -\left|\alpha\right| p-{l p}}r^n 2^{kn}\lesssim1,\nonumber
\end{align*}
where we pick $l\in Z_+$ such that $n-np-s-lp<0$.

Analogously to the case of  $|\alpha|\leq s \neq\gamma_p$, when $|\alpha|= s =\gamma_p$,  repeating the estimate of $\eqref{e7}$ with $np+\left|\alpha\right|p=n$ and $r\geq1$, we have
\begin{align*}
\mathrm{I_2}
& \lesssim \sum_{|\alpha| = s}\sum_{k=1}^{\infty} \int_{A_k }\sup _{0<t<1}t^{-np-\left|\alpha\right| p}\left|\frac{x-x_B}{t}\right|^{-np-\left|\alpha\right|p -l p}\left|\int _{\mathbb{R}^n}a(y)(x_B-y)^\alpha d y\right|^p \omega(x)d x\\\nonumber
&\lesssim \sum_{|\alpha| = s} \sum_{k=1}^{\infty} \int_{A_k}\left|x-x_B\right|^{-np -\left|\alpha\right| p-{l p}}\left| \int_{\mathbb{R}^n} a(y) (x_B-y)^\alpha d y\right|^p \omega(x)d x \\\nonumber
&\lesssim \sum_{|\alpha| =s} \sum_{k=1}^{\infty} \int_{A_k}\left(2^kr\right)^{-np -\left|\alpha\right| p-{l p}}r^{n+\eta pn} w(B(x_B,r))^{-1 } \omega(x)d x \\\nonumber
&\lesssim  \sum_{k=1}^{\infty} \left(2^k\right)^{-n-{l p}}r^{\eta pn} r^{-{l p}}2^{kn} \lesssim1,\nonumber
\end{align*}
where we pick $l\in Z_+$ such that $\eta n-l<0$.

Combining  the estimates of  {Cases I} and {II}, we obtain $\mathrm{S}_2\leq C$. This finishes the proof of Theorem \ref{m2.36}.
\end{proof}
\begin{theorem}\label{l3.5}
Let $\omega\in A_{1}(\rn)$, $0 < p \le 1\leq q<\infty$ with $p\neq q$ and  $s:=[n(\frac{1}{p}-1)]$.  Then
\begin{align*}
h_{\omega,approx}^{p, q, s} (\mathbb{R}^n )=h_w^p (\mathbb{R}^n)
\end{align*}
with  equivalent  norms.
\end{theorem}
To prove Theorem \ref{l3.5}, we need a lemma as follows.
\begin{lemma}\label{l2.11}
\cite{st89} Let $w \in A_{\infty}\left(\rn\right)$ and $0<p \leq 1$.
For each $f \in h^p_\omega(\rn)$, there exist $\left\{a_i\right\}$ of $(p, \infty$, $s, w)$-atoms and $\left\{\lambda_i\right\}$ of real numbers with $\sum_{i=1}^{\infty}\left|\lambda_i\right|^p \leq C\|f\|_{ h^p_\omega}^p$ such that $f=\sum_{i=0}^{\infty}\lambda_i a_i$ in $ h^p_\omega(\rn)$.
\end{lemma}
%
%

\begin{proof}[Proof of Theorem \ref{l3.5}]
For any $f\in h^p_\omega(\rn)$, by Lemma \ref{l2.11}, there exist $\left\{a_i\right\}$ of $(p, \infty$, $s, w)$-atoms and $\left\{\lambda_i\right\}$ of real numbers with $\sum_{i=1}^{\infty}\left|\lambda_i\right|^p \leq C\|f\|_{ h^p_\omega}^p$ such that $f=\sum_{i=0}^{\infty} \lambda_i a_i$ in $ h^p_\omega(\rn)$. Since every local weighted $(p, \infty$, $s, \omega)$-atom is also a $(p, q$, $s, w)$-atom and every  $(p, q$, $s, w)$-atom is also a $(p, q, s,\omega)$-approximate atom (Remark \ref{re1.1}(iv)), therefore
$$\|f\|_{h_{\omega,approx}^{p, q, s}}\leq\left(\sum_{i=0}^{\infty}\left|\lambda_i\right|^p\right)^{1 / p}\leq C\|f\|_{h_w^{p}}<\infty$$
and hence $f\in h_{w,approx}^{p, q, s}$, which implies  $h^p_\omega(\rn)\subset  h_{w,approx}^{p, q, s}(\rn)$.

Conversely, let us now prove $h_{w,approx}^{p, q, s}(\rn)\subset h^p_\omega(\rn)$. For any $f\in h_{w,approx}^{p, q, s}(\rn)$, there exist $\left\{a_i\right\}$ of $(p, q, s,\omega)$-approximate atoms and $\left\{\lambda_i\right\}$ of real numbers such that
\begin{align}\label{eq2.111}
f=\sum_{i=0}^{\infty} \lambda_i a_i,\quad \mathrm{and}\quad
\left(\sum_{i=0}^{\infty}\left|\lambda_i\right|^p\right)^{1 / p}\leq 2\|f\|_{h_{w,approx}^{p, q, s}}.
 \end{align}
Then by \eqref{eq2.111} and Theorem \ref{m2.36}, we have
$$
\begin{aligned}
\|f\|_{h_w^{p}}(\rn)=\|m_\phi f\|_{L_{\omega}^{p}}^p & =\left\|\sup _{0<t<1}\left|\phi_t *\left(\sum_{i=1}^{\infty} \lambda_ia_i\right)\right|\right\|_{L_{\omega}^{p}}^p \\
& =\left\|\sup _{0<t<1}\left|\sum_{i=1}^{\infty} \lambda_i\left(\phi_t*a_i\right)\right|\right\|_{L_{\omega}^{p}}^p \\
& \leq\left\|\sum_{i=1}^{\infty}\left|\lambda_i\right|\left(\sup _{0<t<1}\left|\phi_t *a_i\right|\right)\right\|_{L_{\omega}^{p}}^p \\
& \leq \sum_{i=1}^{\infty}\left|\lambda_i\right|^p\left\|\sup _{0<t<1}\left|\phi_t *a_i\right|\right\|_{L_{\omega}^{p}}^p\\
& \lesssim \left(\sum_{i=0}^{\infty}\left|\lambda_i\right|^p\right)
\lesssim \|f\|_{h_{w,approx}^{p, q, s}}^p,
\end{aligned}
$$
and hence $f\in h_{\omega}^{p}(\rn)$, which implies $ h_{w,approx}^{p, q, s}(\rn)\subset  h^p_\omega(\rn)$.

Combining the above discussion, we finish the proof of  Theorem \ref{l3.5}.
\end{proof}
\section{The molecules of weighted local Hardy spaces }
\begin{definition}\label{d2.11}
 Let $w\in A_{1}(\rn)$, $0<p \leq 1 \leq q \leq \infty$ with $p<q$, $\lambda >n(\frac{q}{p}-1)$ and $s:=[n(\frac{1}{p}-1)]$.   $M$ is a {\it $(p, q, s, \lambda, \omega)$-approximate molecule} if there exists a ball $B(x_{B}, r)$ and $\eta>0$ such that a function $M$ on $\rn$ satisfies $(\mathrm{M}_1)$, $(\mathrm{M}_2)$ and  the following $(\mathrm{M}_3)$ when $r<1$:
\begin{enumerate}
\item [$(\mathrm{M}_1)$] $\|M\|_{L_\omega^q(B)}^q\leq w\left(B\left(x_B, r\right)\right)^{\frac{1}{q}-\frac{1} { p}}$, \\
\item [$(\mathrm{M}_2)$] $ \|M\left|\cdot-x_B\right|^\lambda \|_{L_\omega^q(B^c)}^q\leq r^{\frac{\lambda}{q}} w\left(B\left(x_B, r\right)\right)^{\frac{1}{q}-\frac{1} { p}},$\\
\item [$(\mathrm{M}_3)$]  $\begin{cases}\left|\int_{\rn} M(x) (x-x_B)^{\alpha }d x\right| \leq C\left(\frac{\left|B\right|}{\omega(B)}\right)^{\frac{1}{p}}&\text{if}\quad|\alpha|<\gamma_p\neq s\\\left|\int_{\rn} M(x) (x-x_B)^{\alpha }d x\right| \leq C\left(\frac{\left|B\right|}{\omega(B)}\right)^\beta \omega(B)^\eta&\text{if}\quad |\alpha|=\gamma_p=s\end{cases},$
\end{enumerate}
where $\beta:=\eta+\frac{1}{p}$, or $M$ satisfies $(\mathrm{M}_1)$ and $(\mathrm{M}_1)$ when  $ r\geq1 $.
\end{definition}

Next, we show that a $(p,q,s, \lambda, \omega)$-approximate molecule can be decomposed  into the sum of weighted atoms and a weighted approximate atom with uniformly bounded norm in $h{_{\omega}^p}(\rn)$; see the following Theorems \ref{t3.11} and \ref{d3.66}, the proofs of which are shown later.
\begin{theorem}\label{t3.11}
Let $\omega\in A_{1}(\rn)$, $0<p \leq 1 < q < \infty$, $\lambda >n(\frac{q}{p}-1)$   and   $s:=[n(\frac{1}{p}-1)]$.
If $M$ is a  $(p, q, s, \lambda, \omega)$-approximate molecule, then
$$
M(x)=\sum_{k=0}^{\infty} t_k a_k(x)+\sum_{k=0}^{\infty} s_k b_k(x)+a(x) \quad \mathrm{for} \quad a.e.\quad x\in \rn
$$
and
$$
\sum_{k=0}^{\infty}\left|t_k\right|^p<C  \quad \mathrm{and}\ \ \sum_{k=0}^{\infty}\left|s_k\right|^p<C,
$$
where $\{a_k\}$ and $\{b_k\}$ are $(p, q , s,\omega)$-atoms and $(p, \infty, s, \omega )$-atoms, respectively, and $a$ is a $ (p, q,s,\omega)$-approximate atom.
\end{theorem}

\begin{theorem}\label{d3.66}
Let $\omega\in A_{1}(\rn)$, $0<p \leq 1 \leq q \leq \infty$ with $p<q$, $\lambda >n(\frac{q}{p}-1)$   and     $s:=[n(\frac{1}{p}-1)]$.
If $M$ is a  $(p, q, s, \lambda, \omega)$-approximate molecule, then $\|M\|_{h_\omega^p }\leq C$, where C is a positive constant independent of  $M$.
\end{theorem}
To prove Theorem \ref{t3.11}, we need  Lemma \ref{m3.1}.
\begin{lemma}\label{m3.1}
 Let  $w\in A_{1}(\rn)$, $0<p \leq 1 < q < \infty$,  $\lambda >n(\frac{q}{p}-1)$ and $s:=[n(\frac{1}{p}-1)]$. If $M$ is a $(p, q, s, \lambda, \omega)$-approximate molecule associated  with $B=B(x_B,r)$,
then
\begin{align*}
\left|\int_{\rn} M(x) (x-x_B)^{\alpha }d x\right| \leq C r^{n+|\alpha|}{w\left(B\left(x_B, r\right)\right)}^{-\frac{1}{p}}, \quad |\alpha|\leq \gamma_p,
\end{align*}
where $C$ is positive constant.
\end{lemma}
\begin{proof}
Let $M$ be a $(p, q, s, \lambda, \omega)$-approximate molecule associated with $B=B(x_B,r)$. Split
\begin{align*}
\left|\int_{\rn} M(x) (x-x_B)^{\alpha }d x\right|
&\leq\left|\int_{2B} M(x) (x-x_B)^{\alpha }d x\right|\\
&\quad+\left|\int_{(2B)^c} M(x) (x-x_B)^{\alpha }d x\right|\\
&=:\mathrm{K}_1+\mathrm{K}_2.
\end{align*}
First, let us estimate $\mathrm{K}_1$.   By ($\mathrm{M}_1$) and Definition $\ref{d2.4}$, we have
\begin{align*}
\mathrm{K}_1
&\leq\left|\int_{2B} M(x) (x-x_B)^{\alpha } \omega(x)^{\frac{1}{q}}\omega(x)^{-\frac{1}{q}}d x\right|\\
&\lesssim (2r)^{|\alpha|}\left(\int_{2B} |M(x)|^q \omega(x)dx\right)^{\frac{1}{q}}\left(\int_{2B}\omega(x)^{-\frac{1}{q-1}}dx\right)^{1-\frac{1}{q}}\\
&\lesssim (2r)^{|\alpha|}\omega(B)^{\frac{1}{q}-\frac{1}{p}} \frac{|B|}{\omega(B)^{\frac{1}{q}}}\\
&\lesssim (2r)^{n+|\alpha|}\omega(B)^{-\frac{1}{p}}\\
&\lesssim r^{n+|\alpha|}\omega(B)^{-\frac{1}{p}}
\end{align*}

Next, let us estimate $\mathrm{K}_2.$  Let $A_k :=\{2^kr \leq\left|x-x_B\right| <2^{k+1}r\}$. By ($\mathrm{M}_2$), Definition $\ref{d2.4}$  and Lemma \ref{l3.16}(i), we obtain
\begin{align*}
\mathrm{K}_2
&\leq\left|\int_{(2B)^c} M(x) (x-x_B)^{\alpha }\frac{|x-x_B|^{\frac{\lambda}{q}}}{|x-x_B|^{\frac{\lambda}{q}}}d x\right|\\
&\leq\left(\int_{(2B)^c} |M(x)|^q|x-x_B|^{\lambda}\omega(x) dx\right)^{\frac{1}{q}}\left(\int_{(2B)^c}|x-x_B|^{(\alpha-\frac{\lambda}{q})\frac{q}{q-1}}\omega(x)^{-\frac{1}{q-1}}dx\right)^{1-\frac{1}{q}}\\
&\leq\left(\int_{(2B)^c} |M(x)|^q|x-x_B|^{\lambda} dx\right)^{\frac{1}{q}}\sum_{k=1}^{\infty}\left(\int_{A_k}|x-x_B|^{(\alpha-\frac{\lambda}{q})\frac{q}{q-1}}\omega(x)^{-\frac{1}{q-1}}dx\right)^{1-\frac{1}{q}}\\
&\lesssim r^\frac{\lambda}{q} \omega(B)^{\frac{1}{q}-\frac{1}{p}} \sum_{k=1}^{\infty}(2^kr)^{|\alpha|-\frac{\lambda}{q}}\frac{|B(x_B,2^{k+1}r)|}{ \omega(B(x_B, 2^{k+1}r))^\frac{1}{q}}\\
&\lesssim \sum_{k=1}^{\infty}r^{n+|\alpha|}\omega(B)^{-\frac{1}{p}}2^{k[|\alpha|-\frac{\lambda}{q}+n(1-\frac{1}{q})]}\\
&\lesssim  r^{n+|\alpha|}\omega(B)^{-\frac{1}{p}},
\end{align*}
where the above series converges due to $$n\left(1- \frac{1}{q}\right)+|\alpha|-\frac{\lambda}{q}<0$$ implied by $|\alpha|\leq \gamma_p \leq n(\frac{1}{p}-\frac{1}{q})< \frac{\lambda}{q}$.

Combining  the estimates of $\mathrm{K}_1$ and $\mathrm{K}_2$, we obtain the proof of Lemma \ref{m3.1}.
\end{proof}
\begin{remark}\label{re1.2}
By Lemma \ref{m3.1} and repeating the discussion of Remark \ref{re1.1}(iv), we can also obtain that  ball-control conditions of  $(\mathrm{M}_3)$ can be deduced by $(\mathrm{M}_1)$ and $(\mathrm{M}_2)$ when $r \geq1$.
\end{remark}
\begin{proof}[Proof of Theorem \ref{t3.11}]
The proof is inspired by the classical molecular decompositions for real Hardy spaces (\cite[Theorem 7.16]{jcc85}). Let $M$ be $a$  $(p, q, s, \lambda, \omega)$-approximate molecule  associated to a ball $B=B\left(x_B, r\right) \subset \mathbb{R}^n$.
Let $B_0:=E_0:=B$, $B_k:=B\left(x_B, 2^k r\right)$, $E_k:=B_k \backslash B_{k-1}$ and $M_k(x):=M(x) \chi_{E_k}(x)$ for $k \in \mathbb{N}$. Let $P_k:=\sum_{|\alpha| \leq s } m_\alpha^k \phi_\alpha^k$ to be the polynomial of degree at most $s$, restricted to the set $E_k$. For every $|\alpha| \leq s$, we have
\begin{align}\label{e8}
& \frac{1}{|E_k|}\int_{E_k} P_k(x)\left(x-x_B\right)^\alpha d x  =m_\alpha^k \frac{1}{|E_k|}\int_{E_k} \phi_\alpha^k(x)\left(x-x_B\right)^\alpha d x\\\nonumber
& =m_\alpha^k:=\frac{1}{|E_k|}\int_{E_k} M(x)\left(x-x_B\right)^\alpha d x.
\end{align}
 This is done by choosing the polynomials $\phi_\alpha^k$ to have $\beta$-th moment equal to
$|E_ k |$ when $\beta = \alpha$, and zero otherwise. And noting that
$
\left(2^k r\right)^{|\alpha|}\left|\phi_\alpha^k(x)\right| \lesssim1
$
uniformly in $k$.
From this, we deduce that
\begin{align}\label{e4.1}
\left|P_k(x)\right| \leqslant \sum_{|\alpha| \leqslant s}\left|m_\alpha^k \phi_\alpha^k(x)\right| \lesssim \frac{1}{\left|E_k\right|} \int_{E_k}\left|M(x)\right| d x.
\end{align}
For $|\alpha| \leq s$, by \eqref{e8}, we obtain
\begin{align}\label{e9}
&\nu_\alpha:=\sum_{j=0}^{\infty}\left|E_j\right| m_\alpha^j=\int_{\mathbb{R}^n} M(x)\left(x-x_B\right)^\alpha d x
\end{align}
and
\begin{align}\label{e10}
&N_\alpha^k:=\sum_{j=k+1}^{\infty}\left|E_j\right| m_\alpha^j=\int_{E_k^c} M(x)\left(x-x_B\right)^\alpha d x, \quad k=0,1,2, \ldots.
\end{align}
From \eqref{e9} and \eqref{e10}, we can represent the sum of $P_k$ as
\begin{align}\label{e11}
\sum_{k=0}^{\infty} P_k(x) & =\sum_{|\alpha| \leq s}\left(\sum_{k=1}^{\infty}\left(N_\alpha^{k-1}-N_\alpha^k\right)\left|E_k\right|^{-1} \phi_\alpha^k(x)+\left(\nu_\alpha-N_\alpha^0\right)\left|E_0\right|^{-1} \phi_\alpha^0(x)\right) \\\nonumber
& =\sum_{|\alpha| \leq s}\left(\sum_{k=0}^{\infty} N_\alpha^k\left|E_{k+1}\right|^{-1} \phi_\alpha^{k+1}(x)-\sum_{k=0}^{\infty} N_\alpha^k\left|E_k\right|^{-1} \phi_\alpha^k(x)+\nu_\alpha\left|E_0\right|^{-1} \phi_\alpha^0(x)\right) \\\nonumber
& =\sum_{k=0}^{\infty} \sum_{|\alpha| \leq s} \Phi_\alpha^k(x)+\sum_{|\alpha| \leq s} \nu_\alpha\left|E_0\right|^{-1} \phi_\alpha^0(x),\nonumber
\end{align}
where
\begin{align}\label{e1.122}
\Phi_\alpha^k(x):=N_\alpha^{k}\left[\left|E_{k+1}\right|^{-1} \phi_\alpha^{k+1}(x)-\left|E_k\right|^{-1} \phi_\alpha^k(x)\right], \quad k=0,1,2, \ldots .
\end{align}
Note that \eqref{e11} appears since we do not have the vanishing moment conditions on the molecule.
This allows us to decompose $M$ as follows:
$$
M=\sum_{k=0}^{\infty}\left(M_k-P_k\right)+\sum_{k=0}^{\infty} \sum_{|\alpha| \leq s} \Phi_\alpha^k(x)+\sum_{|\alpha| \leq s} \nu_\alpha\left|E_0\right|^{-1} \phi_\alpha^0(x)=:\mathrm{I_1+I_2+I_3}.
$$

Let us first estimate $\mathrm{I_1}$. For $1 < q<\infty$, by   Lemma \ref{l3.16}(ii), we have

%

\begin{align}\label{e4.14}
\frac{1}{\left|E_k\right|} \int_{E_k}\left|M_k(x)\right| d x\leq\frac{1}{\left|B_k\right|} \int_{B _k}\left|M_k(x)\right| d x \lesssim \omega(B(x_{B}, 2^kr))^{-{\frac{1}{q}}}\left\|M_k\right\|_{L_\omega^q}.
\end{align}
From \eqref{e4.1} and \eqref{e4.14}, we deduce that
\begin{align}\label{e4.15}
\left|P_k(x)\right|^q \lesssim \left( \omega(B(x_{B}, 2^kr))^{-1}\left\|M\right\|_{L_\omega^q}\right)^q
={C}\omega(B(x_{B}, 2^kr))^{-1}\left\|M_k\right\|_{L_\omega^q}^q.
\end{align}
Then, by \eqref{e4.15}, we have
$$
\begin{aligned}
\int_{E_k}\left|P_k(x)\right|^q \omega(x) d x & \lesssim  \int_{E_k}  \omega\left(B( x_B, 2^k r)\right)^{-1}\|M_k\|_{L_\omega^q}^q \omega(x) d x . \\
&\lesssim  \omega\left(B( x_B, 2^k r)\right)^{-1}\|M_k\|_{L_w^q}^q \int_{B_k} \omega(x) d x . \\
& =C\left\|M_k\right\|_{L_\omega^q}^q,
\end{aligned}
$$
which implies that
\begin{align*}
||P_k||_{L_{\omega}^{q}}\lesssim \|M_k\|_{L_\omega^q}
\end{align*}
and
\begin{align}\label{e2.18}
||M_k-P_k||_{L_{\omega}^{q}}\leq \|M_k\|_{L_{\omega}^{q}}+\|-P_k\|_{L_{\omega}^{q}} \lesssim\|M_k\|_{L_\omega^q}.
\end{align}
For $k=0,1,2,\cdots$, by $(\mathrm{M}_1)$ and $(\mathrm{M}_2)$ of Definition \ref{d2.11},  we have
\begin{align}\label{e2.19}
\left(\int_{\rn}|M(x)\chi_{E_k}(x)|^q \omega(x) d x\right)^\frac{1}{q}& =\left(\int_{E_k}|M(x)|^q w(x) d x\right)^\frac{1}{q} \\\nonumber
&\lesssim 2^{-k\frac{\lambda}{q}}\omega\left(B(x_B, r)\right)^{\frac{1}{q}-\frac{1}{p}}.\nonumber
\end{align}
By \eqref{e2.18} and \eqref{e2.19}, we obtain

$$
\left\|M_k-P_k\right\|_{L_\omega^q}\lesssim  2^{-k\frac{\lambda}{q}}\omega\left(B(x_B, r)\right)^{\frac{1}{q}-\frac{1}{p}}.
$$
Therefore, for $k=0,1,2,\cdots$, it follows that
\begin{align}\label{e2.20}
a_k:=C^{-1} 2^{k\frac{\lambda}{q}}\left(M_k-P_k\right),
\end{align}
where $a_k$ is a $\left(p , q, s, \omega\right)$-atom supported on $B_k$,
in other words, $$M_k-P_k=t_k a_k,\quad t_k:=C  2^{-k\frac{\lambda}{q}}. $$
Moreover, we have
\begin{align}\label{e2.21}
\sum_{k=0}^{\infty}\left|t_k\right|^p=C\sum_{k=0}^{\infty} 2^{-kp\frac{\lambda}{q}}\lesssim 1.
\end{align}

 For $\mathrm{I_2}$, let us first estimate $N_\alpha^k$ of $\Phi_{\alpha }^k(x)$.
By \eqref{e10}, \eqref{e4.14}, \eqref{e2.19} and Lemma \ref{l3.16}(i), we have
\begin{align*}
& \left|N_\alpha^k\right| \leq \sum_{j=k+1}^{\infty}\left|E_j\right| \frac{1}{|E_j|} \int_{E_j} \left|M(x)\left(x-x_B\right)^\alpha\right| d x \\
&\qquad \leq \sum_{j=k+1}^{\infty}\left(2^j r\right)^{n+|\alpha|} \frac{1}{|E_j|} \int_{E_j} |M(x)| d x \\
&\qquad \lesssim \sum_{j=k+1}^{\infty}(2^j r)^{n+|\alpha|} \omega\left(B\left(x_B, 2^j r\right)\right)^{-\frac{1}{q}} 2^{- j\frac{\lambda}{q}}\omega\left(B\left(x_B, r\right)\right)^{\frac{1}{q}-{\frac{1}{p}}} \\
&\qquad \lesssim \sum_{j=k+1}^{\infty} 2^{j[n(1- \frac{1}{q})+|\alpha|-\frac{\lambda}{q} ]}  r^{n+|\alpha|} \omega\left(B\left(x_B, r\right)\right)^{-\frac{1}{p}} \\
&\qquad \lesssim2^{   [n(1- \frac{1}{q})+|\alpha|-\frac{\lambda}{q} ]}  r^{n+|\alpha|} \omega\left(B\left(x_B, r\right)\right)^{-\frac{1}{p}},
\end{align*}
where the above series converges due to $$n\left(1- \frac{1}{q}\right)+|\alpha|-\frac{\lambda}{q}<0$$ implied by $|\alpha|\leq \gamma_p \leq n(\frac{1}{p}-\frac{1}{q})< \frac{\lambda}{q}$.

By this and $|\phi_\alpha^k(x)| \lesssim (2^k r)^{-|\alpha|}$, we have
\begin{align*}
& | N_\alpha^k \varphi_{\alpha}^k |\lesssim 2^{k[n(1- \frac{1}{q})+|\alpha|-\frac{\lambda}{q} ]}  r^{n+\alpha} \omega\left(B\left(x_B, r\right)\right)^{-\frac{1}{p}}(2^k r)^{-n-|\alpha|} \\
&\qquad\quad  \lesssim  2^{k(-\frac{n}{q}-\frac{\lambda}{q})}  \omega(B\left(x_B , r\right))^{-\frac{1}{p}}.
\end{align*}
From this, we deduce that
\begin{align}\label{2.333}
\Phi_{\alpha }^k(x)=s_k b_k,
\end{align}
where  $b_k$ is a $\left(p, \infty, s, \omega \right)$ atom  and
$s_k:=C_{\omega,n}  2^{-k(\frac{n}{q}+\frac{\lambda}{q})} $. Therefore, \begin{align}\label{2.336}
\sum_{k=0}^{\infty}s_k^p=\sum_{k=0}^{\infty} C_{\omega, n }2^{-kp(\frac{n}{q}+\frac{\lambda}{q})} \lesssim 1.
\end{align}

Now, for  $\mathrm{I_3}$, let us  show  that $a(x):= \sum_{|\alpha| \leq s} \nu_\alpha\left|E_0\right|^{-1} \phi_\alpha^0(x)$  is a multiple of a $ (p, q,s,\omega)$-approximate atom.
By Proposition \ref{m3.1}, we obtain
\begin{align*}
\left|\nu_\alpha\right|=\left|\int_{\mathbb{R}^n} M(x)\left(x-x_B\right)^\alpha d x\right| \leq r^{n+|\alpha|}{\omega\left(B\left(x_B, r\right)\right)}^{-\frac{1}{p}}
\end{align*}
From this and $|\phi_\alpha^0(x)| \lesssim ( r)^{|-\alpha|}$, it follows that
\begin{align*}
 \left\|\sum_{|\alpha| \leq s}\nu_\alpha|E_0 |^{-1}\phi_\alpha^0(x)\right\|_{L_\omega^q}&\lesssim \sum_{|\alpha| \leq s} r^{-(|\alpha|+n)} r^{|\alpha|+n}  {\omega\left(B\left(x_B, r\right)\right.}^{\frac{1}{q}-\frac{1}{p}} \\
& \lesssim \omega\left(B\left(x_B , r\right)\right)^{\frac{1}{q}-\frac{1}{p}}.
\end{align*}

It remains to show the moment conditions on $a$, which follows immediately from $\left(M_3\right)$, since by the choice of $\nu_\alpha$ and $\phi_\alpha^0$, the moments of $a$ are the same as those of $M$. Indeed, for $|\beta| \leq s$,
\begin{align*}
&\int_{\rn} a{(x)}\left(x-x_B\right)^\beta d x=\sum_{|\alpha| \leq s} \nu_\alpha\left(\left|E_0\right|^{-1} \int_{E_0} \phi_\alpha^0(x)\left(x-x_B\right)^\beta d x\right)\\
& \qquad\qquad\qquad\qquad\quad=\nu_\beta=\int_{\mathbb{R}^n} M(x)\left(x-x_B\right)^\beta d x.
\end{align*}
Thus $a$ is a multiple of a  $ (p, q,s,\omega)$-approximate atom.

Combining this, (\ref {e2.20}), (\ref{e2.21}), (\ref{2.333}) and (\ref{2.336}), we finished the proof of Theorem \ref{t3.11}.

\end{proof}
\begin{proof}[Proof of Theorem \ref{d3.66}]
Suppose $M$ is $a$  $(p, q, s, \lambda, \omega)$-approximate molecule  associated to a ball $ B=B\left(x_B, r\right) \subset \mathbb{R}^n$. By Theorem \ref {t3.11}, we have
\begin{align}\label{f1.1}
M(x)=\sum_{k=0}^{\infty} t_k a_k(x)+\sum_{k=0}^{\infty} s_k b_k(x)+a(x) \quad \mathrm{for} \quad a.e.\quad x\in \rn
\end{align}
and
\begin{align}\label{f1.12}
\sum_{k=0}^{\infty}\left|t_k\right|^p\lesssim1 \text  { and }\ \ \sum_{k=0}^{\infty}\left|s_k\right|^p\lesssim1,
\end{align}
where $\{a_k\}$ and $\{b_k\}$ are $(p, q , s,\omega)$-atoms and $(p, \infty, s, \omega )$-atoms, respectively, and $a$ is a $ (p, q,s,\omega)$-approximate atom.

Let  $\phi \in \mathcal{S}(\rn)$ with $\int_{\rn} \phi(x)dx \neq 0$.
For any $x\in \rn $, by (\ref {f1.1}) and  Fatou's lemma, we obtain
\begin{align*}
&\left|m_\phi (M)(x)\right|=\sup_{0<t<1}\left|\int_{\rn}\phi_t(x-y)M(y)dy\right|\\
&=\sup_{0<t<1}\left|\int_{\rn}\phi_t(x-y)\left(\sum_{k=0}^{\infty} t_k a_k(y)+\sum_{k=0}^{\infty} s_k b_k(y)+a(y)\right)dy\right|\\
&=\sup_{0<t<1}\left|\int_{\rn} \varliminf_{N\to \infty}\sum_{k=0}^{N} t_k \phi_t(x-y) a_k(y)+\varliminf_{N\to \infty}\sum_{k=0}^{N} s_k \phi_t(x-y)b_k(y)+\phi_t(x-y)a(y) dy\right|\\
& \leq \sup_{0<t<1}\varliminf_{N\to \infty}\left|\int_{\rn} \sum_{k=0}^{N} t_k \phi_t(x-y) a_k(y)dy\right|+\sup_{0<t<1}\varliminf_{N\to \infty}\left|\int_{\rn} \sum_{k=0}^{N} s_k \phi_t(x-y)b_k(y)dy\right|\\
& \qquad+ \sup_{0<t<1}\left|\int_{\rn}\phi_t(x-y)a(y) dy\right|\\
& \leq \sup_{0<t<1}\varliminf_{N\to \infty}\sum_{k=0}^{N} |t_k| \left|\int_{\rn} \phi_t(x-y) a_k(y)dy\right|+\sup_{0<t<1}\varliminf_{N\to \infty}\sum_{k=0}^{N} |s_k|\left|\int_{\rn} \sum_{k=0}^{N} s_k \phi_t(x-y)b_k(y)dy\right|\\
& \qquad+ \sup_{0<t<1}\left|\int_{\rn}\phi_t(x-y)a(y) dy\right|\\
& \leq \sup_{0<t<1}\lim_{N\to \infty}\sum_{k=0}^{N} |t_k| \left|\int_{\rn} \phi_t(x-y) a_k(y)dy\right|+\sup_{0<t<1}\lim_{N\to \infty}\sum_{k=0}^{N} |s_k|\left|\int_{\rn} \sum_{k=0}^{N} s_k \phi_t(x-y)b_k(y)dy\right|\\
& \qquad+ \sup_{0<t<1}\left|\int_{\rn}\phi_t(x-y)a(y) dy\right|\\
& \leq \sup_{0<t<1}\sum_{k=0}^{\infty} |t_k|\left| \phi_t\ast a_k(x)\right|+\sup_{0<t<1}\sum_{k=0}^{\infty} |s_k|\left| \phi_t\ast b_k(x)\right| + \sup_{0<t<1}\left|\phi_t\ast a(x)\right|\\
& \leq \sum_{k=0}^{\infty} |t_k|m_\phi a_k(x)+\sum_{k=0}^{\infty} |s_k|m_\phi b_k(x)+m_\phi a(x).
\end{align*}
From this, $0<p\leq1$, (\ref{f1.12}) and Theorem \ref{m2.36},  it follows that
\begin{align*}
&\|M\|_{h_w^{p}}^p(\rn)=\int_{\rn}\left[m_\phi (M)(x)\right]^p\omega(x)dx\\
&\leq\int_{\rn} \sum_{k=0}^{\infty} |t_k|^p\left[m_\phi a_k(x)\right]^p\omega(x)dx+\int_{\rn} \sum_{k=0}^{\infty} |s_k|^p\left[m_\phi b_k(x)\right]^p\omega(x)dx+\int_{\rn}\left[m_\phi a(x)\right]^p\omega(x)dx\\
&\leq \sum_{k=0}^{\infty} |t_k|^p\|a_k\|_{h_w^{p}}^p+\sum_{k=0}^{\infty} |s_k|^p\|b_k\|_{h_w^{p}}^p+\|a\|_{h_w^{p}}^p\\
&\lesssim1,
\end{align*}
which  completes the proof of Theorem \ref{d3.66}.
\end{proof}
\section{$T_{\delta}^{\mu} $: $h_{\omega}^{p}(\rn)\rightarrow h_{\omega}^{p}(\rn)$\label{s2}}

This section is devoted to show  the boundedness of inhomogeneous  Calder\'on-Zygmund operators $T_{\delta}^{\mu} $  on weighted local Hardy  spaces $h_{\omega}^{p}(\rn)$.
\begin{definition}
\cite{dhz21} $K(x,y)$ is said to be an {\it inhomogeneous Calder\'on-Zygmund  kernel} if it is a locally integrable function defined on $\mathbb{R}^{ n}\times\mathbb{R}^{ n}$ away from the diagonal $x = y$ and satisfies the following conditions: if there exist $\mu>0$, $0<\delta \leq 1$ and a constant $C> 0$,
\begin{align}\label{e2.663}
|K(x, y)| \leq C \min \left\{\frac{1}{|x-y|^n}, \frac{1}{|x-y|^{n+\mu}}\right\}, \quad x \neq y;
\end{align}
\begin{align}\label{e2.66}
|K(x, y)-K(x, z)|+|K(y, x)-K(z, x)| \leq C\frac{|y-z|^{\delta}}{|x-z|^{n+\delta}},\quad |x-z| \geq 2|y-z|.
\end{align}
 Furthermore, $T_{\delta}^{\mu} $ is said to be an {\it inhomogeneous Calder\'on-Zygmund operator}  if  the following property is satisfied:
   $T_{\delta}^{\mu} $ is associated to an $(\mu,\delta)-$ inhomogeneous standard kernel given (formally) by
 \begin{equation*}
\langle T_{\delta}^{\mu} f,g \rangle: = \iint_{{\rn}\times{\rn}
} K(x,y)f(y)g(x)dydx,
\end{equation*}
for all $f, \,g \in \mathcal{S}\left(\mathbb{R}^n\right)$ with disjoint supports.
\end{definition}
Lighted by \cite[Definition 5.1]{2wqqk}, we introduce the following Definition  \ref{d2.9}.
\begin{definition}\label{d2.9}
 Let $N \in \mathbb{Z}_{+}$ and $\omega \in A_1(\rn)$. We denote by $L_{c, N}^q\left(\mathbb{R}^n; \omega\right)$ the space of all compactly supported functions $g \in L_{\omega}^q\left(\mathbb{R}^n\right)$ such that $\int_{\rn} g(x) x^\alpha dx=0$ for all $|\alpha| \leq N $. For such an $\alpha$, define $T^*\left(x^\alpha\right)$ in the distributional sense by
\begin{align}\label{d2.7}
\left\langle T^*\left(x^\alpha\right), g\right\rangle=\left\langle x^\alpha, T(g)\right\rangle=\int_{\mathbb{R}^n} x^\alpha T g(x) d x, \quad \forall g \in L_{c, N}^q\left(\mathbb{R}^n; \omega\right).
\end{align}
The space $L_{c, s}^q\left(\mathbb{R}^n; \omega\right)$ corresponds to multiples of local weighted $(p, q, s, \omega)$-atom in $h_{\omega}^p\left(\mathbb{R}^n\right)$. That $T^*\left(x^\alpha\right)$ is well-defined by \eqref{d2.7} has been stated for standard Calder\'on-Zygmund  operators with $q=2$ in \cite[p.\,23]{mc97}.
\end{definition}

\begin{theorem}\label{t4.11}

Let $\omega \in A_1(\rn)$, $0<p \leq 1$, $\gamma_p:=n(\frac{1}{p}-1)$, $s:=[\gamma_p]$  and $1< q < \infty$. $T_{\delta}^{\mu} $  be an inhomogeneous Calder\'on-Zygmund operator. Then, for local weighted $(p, q, s, \omega)$-atom $a$ supported in $B=B(x_B,r)$, $T_{\delta}^{\mu} a$ is a $(p, q, s, \lambda, \omega)$-approximate molecule provide that  $\lambda>n(\frac{q}{p}-1)$, $\min \{\mu, \delta\}>\gamma_p$ and there exist $C>0$ and $\eta>0$ for any multi-index $\alpha$ with $|\alpha| \leq s$,

\begin{align}\label{t5.16}
f=(T_{\delta}^{\mu})^\ast\left[(\cdot-x_B)^\alpha\right] \quad \text { satisfies }
&\quad\left(\int_B\left|f(x)-P_B^{s}(f)(x)\right|^{q^{\prime}} \omega(x)^{-{q^{\prime}/q}}d x\right)^{1 /q^{\prime}} \\
&\leq \begin{cases}C{w(B(x_B,r))}^{-\frac{1}{q}}|B(x_B,r)|^{\frac{1}{p}}& \text { if }|\alpha|<\gamma_p\neq s \\C|B|^{{\eta+\frac{1}{p}}}{w(B(x_B,r))}^{-\frac{1}{q}} & \text { if }|\alpha|=\gamma_p=s\end{cases},\nonumber
\end{align}
where $P_B^{s}(f)$ is the polynomial of degree $\leq s$ that has the same moments as $f$ over $B$ up to order $s$.
\end{theorem}\label{t4.1}

\begin{theorem}\label{t4.101}
Let $\omega \in A_1(\rn)$, $0<p \leq 1$, $T_{\delta}^{\mu} $ be an inhomogeneous Calder\'on-Zygmund operator satisfying $\min \{\mu, \delta\}>\gamma_p:=n(\frac{1}{p}-1)$, $s:=[\gamma_p]$ and there exist $C>0$  and $\eta>0$ such that \eqref {t5.16} holds true for any ball $B \subset \rn $ and any $|\alpha|\leq s.$ If $T_{\delta}^{\mu} $ is bounded on $L^2(\rn)$, then $T_{\delta}^{\mu} $ can be extended to a bounded operator from $h_w^p\left(\mathbb{R}^n\right)$ to itself.
\end{theorem}

\begin{proof}[Proof of Theorem \ref{t4.11}]
We assume local weighted $(p,q,s ,\omega)$ -atom a is supported in $B(x_{B}, r)$. We shall show
$T_{\delta}^{\mu}a$ is a $(p, q, s, \lambda, \omega)$-approximate molecule associated with $2B=B(x_B,2r)$, where $1\leq q\leq \infty$ and $\lambda>n(\frac{q}{p}-1)$. By the $L_{\omega} ^q(\rn)$-boundedness of $T_{\delta}^{\mu}$ (see \cite[Theorem 1.9]{jcjubuhd23}) and $(A_2)$ condition of atom $a$, we obtain
\begin{align}\label{4.11}
\lf(\int_{2B}|T_{\delta}^{\mu}a(x)|^{q}w(x)dx\r)^\frac{1}{q}\lesssim \|a\|_{L_{\omega}^q}\lesssim w(B(x_{B}, r))^{\frac{1}{q}-\frac{1}{p}}.
\end{align}
Form this, condition  $(\mathrm{M}_1)$  holds true for $T_{\delta}^{\mu}a$.\\

Let us now verify   $(\mathrm{M}_2)$ condition of $T_{\delta}^{\mu}a$ by two cases: $r\geq1$ and $r<1$.

\textbf{Case I : prove $(\mathrm{M}_2)$ under  $r\geq1$ }.
If $|x-x_{B}|\geq 2r$,  by \eqref{e2.66} and  Definition \ref{d2.4}, we obtain
\begin{align}\label{4.011}
&|T_{\delta}^{\mu}a(x)\leq\int_B |K(x,y)|a(y)dy\leq C|x-x_B|^{-n-\mu}\int_B |a(y)|dy\\\nonumber
&\qquad\quad\lesssim|x-x_B|^{-n-\mu}\left(\int_B |a(y)|^q \omega(y)dy\right)^\frac{1}{q}\left(\int_B  \omega(y)^{-\frac{1}{q-1}}dy\right)^{1-\frac{1}{q}}\\\nonumber
&\qquad\quad\lesssim|x-x_B|^{-n-\mu}\omega(B)^{\frac{1}{q}-\frac{1}{p}}\frac{|B|}{\omega(B)^\frac{1}{q}}\\\nonumber
& \qquad\quad\lesssim|x-x_B|^{-n-\mu}\frac{|B|}{\omega(B)^{\frac{1}{p}}}\nonumber.
\end{align}

Since $\gamma_p<\mu$, then we have $n(\frac{q}{p}-1)<q(n+\mu)-n$. Therefore we  can pick a $\lambda>0$ such that
$$n\left(\frac{q}{p}-1\right)<\lambda<q(n+\mu)-n.$$ Let $A_k :=\{2^kr \leq\left|x-x_B\right| < 2^{k+1}r\}$, $k=1,2,3, \cdots$. By condition $(\mathrm{M}_2)$ and Lemma \ref{l3.16}(i) with $\omega \in A_1(\rn)$, we have
\begin{align}\label{4.112}
&\int_{|x-x_{B}|\geq 2r}|T_{\delta}^{\mu}a(x)|^{q}|x-x_B|^{\lambda}\omega(x)dx \nonumber
\lesssim\lf(\frac{|B|}{\omega(B)^{\frac{1}{p}}}\r)^q\int_{(2B)^c}|x-x_B|^{\lambda-q(n+\mu)}\omega(x)dx \nonumber\\
&\lesssim\lf(\frac{|B|}{\omega(B)^{\frac{1}{p}}}\r)^q\sum_{k=1}^{\infty}\int_{A_k}|x-x_B|^{\lambda-q(n+\mu)}\omega(x)dx
\lesssim\lf(\frac{|B|}{\omega(B)^{\frac{1}{p}}}\r)^q\sum_{k=1}^{\infty}(2^kr)^{\lambda-q(n+\mu)}(2^{kn})\omega(B)\nonumber\\
&\lesssim\sum_{k=1}^{\infty}r^{\lambda-qn-q\mu+qn}2^{k(n+\lambda-qn-q\mu)}\omega (B)^{-\frac{q}{p}}\omega(B)
\lesssim\sum_{k=1}^{\infty}r^{\lambda}2^{k(n+\lambda-qn-q\mu)}\omega (B)^{1-\frac{q}{p}}\nonumber\\
&\lesssim r^{\lambda}\omega (B)^{1-\frac{q}{p}}.
\end{align}

\textbf{Case II : prove $(\mathrm{M}_2)$ under  $r<1$ }.  Since $\gamma_p<\delta$, then we have $n\left(\frac{q}{p}-1 \right)<q(n+\delta)-n$. Therefore we  can pick a $\lambda>0$ such that
$$n\left(\frac{q}{p}-1\right)<\lambda<q(n+\delta)-n.$$

By the vanishing moment condition of the atom, \eqref{e2.66},  Lemma \ref{l3.16} with $\omega \in A_1(\rn)$, condition $(\mathrm{M}_2)$  and general Minkowski's
inequality, we have
\begin{align}\label{4.112}
&\int_{(2B)^c}|T_{\delta}^{\mu}a(x)|^{q}|x-x_B|^{\lambda}\omega(x)dx\\\nonumber
&=\sum_{k=1}^{\infty}\int_{A_k}\left|\int_B|K(x,y)-K(x,x_B)|a(y)dy\right|^q|x-x_B|^{\lambda}\omega(x)dx\\\nonumber
&\leq\sum_{k=1}^{\infty}(2^{k+1}r)^\lambda\left\{\int_B |a(y)|\left[\int_{A_k}\left|K(x,y)-K(x,x_B)\right|^q\omega(x)dx\right]^\frac{1}{q}dy\right\}^q\\\nonumber
&\lesssim\sum_{k=1}^{\infty}(2^{k}r)^\lambda\left\{\int_B |a(y)|\left[\int_{A_k}\frac{|y-x_B|^{q\delta}}{|x-x_B|^{q(n+\delta)}}\omega(x)dx\right]^\frac{1}{q}dy\right\}^q\\\nonumber
&\lesssim\sum_{k=1}^{\infty}(2^{k}r)^{\lambda}\frac{r^{q\delta }}{(2^{k}r)^{(n+\delta )q}}2^{kn}\omega(B)r^{nq}\omega(B)^{-\frac{q}{p}}\\\nonumber
&\lesssim\sum_{k=1}^{\infty}(2^{k}r)^{\lambda} 2^{-k(n+\delta)q}\omega(B)^{1-\frac{q}{p}}\\\nonumber
&\lesssim\sum_{k=1}^{\infty} 2^{k(\lambda+n)}2^{-k(n+\delta)q}r^{\lambda}\omega(B)^{1-\frac{q}{p}}\\\nonumber
&\lesssim\sum_{k=1}^{\infty}2^{k(\lambda+n-nq-q\delta)}r^{\lambda}\omega(B)^{1-\frac{q}{p}}\\\nonumber
&\lesssim r^{\lambda}\omega (B)^{1-\frac{q}{p}}\nonumber.
\end{align}

Therefore,  condition $(\mathrm{M}_2)$   holds true  for $r<1$. Combining the estimates of {cases I and II}, we finally obtain that $T_{\delta}^{\mu}a$ satisfies condition  $ (\mathrm{M}_2)$.

Finally, let us  verify  $(\mathrm{M}_3)$ when $r<1$. By \eqref{d2.7}, \eqref{t5.16} and setting $f=(T_{\delta}^{\mu})^*\left[\left(\cdot-x_B\right)^\alpha\right]$, we have
\begin{align}\label{4.13}
\left|\int_{\rn} T_{\delta}^{\mu} a(x)\left(x-x_B\right)^\alpha d x\right| = \nonumber |\langle f, a\rangle| &\leq\left(\int_B\left|f(x)-P_B^{s}(f)(x)\right|^{q^{\prime}}\omega(x)^{-{q^{\prime}/q}} d x\right)^{1 / q^{\prime}}\|a\|_{L_{\omega}^q(B)} \\
&\lesssim \begin{cases}\left(\frac{\left|B\right|}{\omega(B)}\right)^{\frac{1}{p}}& \text { if }|\alpha|<\gamma_p\neq s \\\left(\frac{\left|B\right|}{\omega(B)}\right)^\beta \omega(B)^\eta & \text { if }|\alpha|=\gamma_p=s\end{cases},
\end{align}
where $\beta=\eta+\frac{1}{p}$.

Thus, condition $(\mathrm{M}_3)$   holds true  for $r<1$. Therefore, $T_{\delta}^{\mu}a(x)$ is a $(p, q, \eta, \lambda, \omega)$-approximate molecule and we finish the proof of Theorem  \ref{t4.11}.

\end{proof}
\begin{proof}[Proof of Theorem \ref{t4.101}]
 Pick a function  $f\in h^p_\omega(\rn)\cap L^2(\rn)$, by \cite[Theorem 4.8 and Remark 4.9]{YSB20}, we know that there exist local weighted $(p,2,s,\omega)$-atoms $\{a_j\}$ and $\{ \lambda_j\} \subset \mathbb{C} $ such that  $f=\sum_{j=1}^{\infty} \lambda_j a_j$  holds true both in $\mathcal{S}^{\prime}(\rn)$ and $L^2(\rn)$, and
\begin{align}\label{eq1.666}
\sum_{j=1}^{\infty}|\lambda_j|^p\sim \| f\|_{h_w^{p}}^p.
\end{align}

Then, by  the $L^2(\rn)$-boundedness of inhomogeneous  Calder\'on-Zygmund operator $T_{\delta}^{\mu}$, we have
\begin{align*}
T_{\delta}^{\mu} f=\sum_{j=1}^{\infty} \lambda_jT_{\delta}^{\mu} a_j  \quad  \mathrm{holds}\,\, \mathrm{true} \,\, \mathrm{both} \,\, \mathrm{in} \,\,\mathcal{S}^{\prime}(\rn) \,\,\mathrm{and} \,\,L^2(\rn).
\end{align*}

Let $\phi \in \mathcal{S}(\rn)$ with $\int_{\rn}\phi(x)dx\neq0$. Then, for any $x\in{ \rn}$, we further obtain

\begin{align}\label{fz6}
& m_\phi( T_{\delta}^{\mu}f)(x)\leq \sum_{j=1}^{\infty}|\lambda_j| m_\phi T_{\delta}^{\mu}a_j(x).\nonumber
\end{align}
From this, $0<p\leq1$, Theorem \ref{t4.11}, Theorem \ref{d3.66} and \eqref{eq1.666},  it follows that
\begin{align}
&\|T_{\delta}^{\mu}f\|_{h_w^{p}}^p=\|m_\phi( T_{\delta}^{\mu}f)\|_{L_w^{p}}^p\leq \sum_{j=1}^{\infty}|\lambda_j|\|T_{\delta}^{\mu} a_j\|_{h_w^{p}}^p\lesssim\|f\|_{h_w^{p}}^p,\nonumber
\end{align}
which together with the fact that ${h_w^{p}}(\rn)\bigcap L^2(\rn)$ is dense in ${h_w^{p}}(\rn)$ (see \cite[Corollary 4.10]{YSB20}),  we further obtain
\begin{align*}
\|T_{\delta}^{\mu}f\|_{h_w^{p}}^p\lesssim\|f\|_{h_w^{p}}^p \quad \mathrm{for}\,\, \mathrm{any} \,\,f \in {h_w^{p}}(\rn).
\end{align*}
We finished the proof of Theorem  \ref{t4.101}.
\end{proof}

\bigskip\medskip

\noindent Haijing Zhao, Xuechun Yang and Baode Li (Corresponding author),
\medskip

\noindent College of Mathematics and System Sciences\\
 Xinjiang University\\
 Urumqi, 830017\\
P. R. China
\smallskip

\noindent{E-mail }:\\
\texttt{1845423133@qq.com} (Haijing Zhao)\\
\texttt{2760978447@qq.com} (Xuechun Yang)\\
\texttt{baodeli@xju.edu.cn} (Baode Li)\\
\bigskip \medskip


\begin{thebibliography}{30}



\bibitem{jcjubuhd23}
X. Chen and J. Tan, The atomic characterization of weighted local Hardy spaces and its applications, preprint, arXiv: 2306.01441, 2023.
\vspace{-0.3cm}

\bibitem{2wqqk}
G. Dafni, C. H. Lau, T. Picon and C. Vasconcelos, Inhomogeneous cancellation conditions and Calder\'on-Zygmund  type operators on $h^p$, Nonlinear Analysis, 2022, 225: 113110.
\vspace{-0.3cm}



\bibitem{dhz21}
W. Ding, Y. S. Han and Y. P. Zhu, Boundedness of singular integral operators on local Hardy spaces and dual spaces, Potential Anal. 2021, 55: 419-441.
\vspace{-0.3cm}

\bibitem{H1972}
C. Fefferman and E. M. Stein, $H^p$ spaces of several variables.  Acta Math. 1972, 129: 137-193 .
\vspace{-0.3cm}

\bibitem{jcc85}
J. Garc\'ia-Cuerva  and J. L. R. De Francia, Weighted Norm Inequalities and Related Topics, North-Holland Math. Studies, vol. 116, Amsterdam, 1985.
\vspace{-0.3cm}

\bibitem{g79}
D. Goldberg , A local version of real Hardy spaces, Duke Math. J. 1979, 46(1): 27-42.
\vspace{-0.3cm}







\bibitem{mc97}
Y. Meyer, R. Coifman, Wavelets: Calder\'on-Zygmund and multilinear operators, Cambridge University Press, 1997.
\vspace{-0.3cm}

\bibitem{hb81}
B. H. Qui, Weighted Hardy spaces, Math. Nachr. 1981, 103(1): 45-62.
\vspace{-0.3cm}

\bibitem{nwf2}
 T. Schott, Pseudodifferential operators in function spaces with exponential weights. Math. Nachr. 1999, 200(1): 119-149.
\vspace{-0.3cm}


\bibitem{ff1993}
E. M. Stein,  Harmonic Analysis: Real-variable Methods, Orthogonality, and Oscillatory Integrals. Princeton Univ.   Press, Princeton, NJ, 1993.
\vspace{-0.3cm}

\bibitem{st89}
J. O. Str\"omberg and A. Torchinsky, Weighted Hardy Spaces, vol. 1381 Springer, 2006.
\vspace{-0.3cm}


\bibitem{NWF}
 M. Taylor, Pseudodifferential Operators and Linear PDE. Birkh\"auser Boston, 1991.
\vspace{-0.3cm}

\bibitem{at86}
A. Torchinsky, Real-variable Methods in Harmonic Analysis, Academic Press, New York, 1986.
\vspace{-0.3cm}

\bibitem{YSB20}
F. Wang, D. Yang and S. Yang, Applications of Hardy spaces associated with ball quasi-Banach function spaces, Results Math. 2020, 75(1): 26.
\vspace{-0.3cm}



\end{thebibliography}
\end{document}